\documentclass[12pt,a4paper]{amsart}
\usepackage{url}
\usepackage{amsmath,amsthm,amsfonts,amssymb,latexsym}

\newtheorem{theorem}{Theorem}[section]
\newtheorem{proposition}[theorem]{Proposition}
\newtheorem{lemma}[theorem]{Lemma}

\theoremstyle{definition}
\newtheorem{definition}[theorem]{Definition}

\newcommand{\U}{\mathcal U}
\newcommand{\w}{\omega}

\newcommand{\IP}{\mathbb P}

\newcommand{\B}{\mathcal{B}}

\newcommand{\F}{\mathcal{F}}

\newcommand{\V}{\mathcal{V}}

\newcommand{\vid}{\hat{\ \ }}
\newcommand{\bigvid}{\hat{\ \ }}

\newcommand{\uhr}{\upharpoonright}

\newcommand{\name}[1]{\dot{#1}}
\newcommand{\la}{\langle}
\newcommand{\ra}{\rangle}

\newcommand{\forces}{\Vdash}

\newcommand{\hot}{\mathfrak}
\newcommand{\W}{\mathcal W}

\newcommand{\nothing}[1]{}

\title[Hurewicz spaces in the Laver model]{Products of Hurewicz spaces in the Laver model}

\author{Du\v{s}an Repov\v{s}  and Lyubomyr Zdomskyy}

\address{Faculty of Education, and Faculty of Mathematics and Physics,
University of Ljubljana,  Ljubljana, Slovenia 1000.}
\email{dusan.repovs@guest.arnes.si}
\urladdr{http://www.fmf.uni-lj.si/\~{}repovs/index.htm}

\address{Kurt G\"odel Research Center for Mathematical Logic,
University of Vienna, W\"ahringer Stra\ss e 25, A-1090 Wien,
Austria.}
\email{lzdomsky@gmail.com}
\urladdr{http://www.logic.univie.ac.at/\~{}lzdomsky/}


\begin{document}
\maketitle

\begin{abstract}
This article is devoted to the interplay between forcing with fusion
and  combinatorial covering properties. We illustrate this interplay by proving
that in the Laver model
 for the consistency of the Borel's conjecture,
the product of any two metrizable spaces
with the Hurewicz property has the Menger property.
\end{abstract}

\section{Introduction}

A topological space
$X$ has the  \emph{Menger} property (or, alternatively, is a Menger space)
 if for every sequence $\la \U_n : n\in\omega\ra$
of open covers of $X$ there exists a sequence $\la \V_n : n\in\omega \ra$ such that
each $\V_n$ is a finite subfamily of $\U_n$ and the collection $\{\cup \V_n:n\in\omega\}$
is a cover of $X$. This property was introduced by  Hurewicz, and the current name
(the Menger property) is used because Hurewicz
proved in   \cite{Hur25} that for metrizable spaces his property is equivalent to
one  property of a base considered by Menger in \cite{Men24}.
If in the definition above we additionally require that $\{\cup\V_n:n\in\w\}$
is a \emph{$\gamma$-cover} of $X$
(this means that the set $\{n\in\w:x\not\in\cup\V_n\}$ is finite for each $x\in X$),
then we obtain the definition of the \emph{Hurewicz  property}  introduced
in \cite{Hur27}. Each $\sigma$-compact space is obviously a Hurewicz space, and Hurewicz spaces have the
Menger property. Contrary to a conjecture of Hurewicz
the class of  metrizable spaces having the Hurewicz property
 appeared  to be much wider than the class of $\sigma$-compact spaces \cite[Theorem~5.1]{COC2}.
The properties of Menger and Hurewicz are classical examples of
combinatorial covering properties of topological spaces which are
nowadays also called selection principles. This is a
growing area of general topology, see, e.g.,  \cite{Tsa07}. For
instance, Menger and Hurewicz spaces found applications in such
areas as forcing \cite{ChoRepZdo15}, Ramsey theory in algebra
\cite{Tsa??},  combinatorics of discrete subspaces \cite{Aur10}, and
Tukey relations between hyperspaces of compacts \cite{GarMedZdo??}.

 Even before the era of
combinatorial covering properties, there was a lot of activity around the study
of special sets of reals. These studies resolved many classical
questions in general topology and measure theory. As a result,
information about special sets of reals is included in standard
topology textbooks, such as Kuratowski's Topology. The most
influential survey on special sets of reals is, probably, Miller's
chapter \cite{Mil84} in the Handbook of Set-Theoretic Topology. The
most recent monograph on this topic is written by Bukovsky, see
\cite{Buk11}. It complements nicely the classical book \cite{BarJud95} of
Bartoszynski and Judah. This theory still finds interesting applications
in general topology, see, e.g., \cite{HruMil??} for the interplay between
$\lambda$-sets and homogeneity.

The theory of combinatorial covering properties, which originated in
\cite{COC2, Sch96}, can be thought of as a continuation  to that of
special sets of reals, with emphasis on the behaviour of their open
or Borel covers. Some  combinatorial covering properties including
the Menger and Hurewicz ones are about 15 years older than G\"odel's
works on $L$ and  40 years  older than the method of forcing, and
they were introduced in the areas of topology where set-theoretic
methods are quite rare even nowadays. E.g., the original idea behind
the Menger   property, as it is explicitly stated in the first
paragraph of \cite{Men24}, was an application in dimension theory.
However, since at least \cite{Lav76} it has become clear that the
combinatorial covering properties are strongly influenced by
axiomatics and hence can be studied with the  help of forcing, see,
e.g., \cite{BarDow12, ChoHru???, ChoRepZdo15, MilTsa10, SchTal10}
for the more recent works along these lines. There are  equivalences
among statements from disciplines with diverse origins (Ramsey
theory, game theory, function spaces and convergence, topological
groups, dimension theory, covering properties, combinatorial set
theory, forcing, hyperspaces, filters, etc.) with combinatorial
covering properties. Even though not all of these have  found
non-trivial applications  so far  (by translating into the other
fields,  via an equivalence, the results known for combinatorial
covering properties), they are offering an alternative point of view
onto the known properties and thus enhance their understanding.
E.g.,  it is shown in \cite{ChoRepZdo15} that a Mathias forcing
associated to a filter $\F$ on $\w$ does not add dominating reals
iff $\F$ is Menger as a subspace of $2^\w$, thus demonstrating that
this property of filters is topological and in this way answering
some questions for which it was unclear how the ``standard''
approaches in this area  can be used.

One of the basic questions about a topological property is whether
it is preserved by various kinds of products in certain classes of
spaces. As usually, the preservation results may be divided into
\emph{positive}, asserting that properties under consideration are
preserved by products (e.g., the classical Tychonoff theorem), and
\emph{negative} which are typically some constructions of spaces
possessing certain property whose product fails to have it (e.g.,
the folklore fact that the Lindel\"of property is not preserved even
by squares, as witnessed by the Sorgenfrei line).
 In case of combinatorial covering properties we know that  the strongest
 possible  negative result is consistent: Under CH there exist
 $X,Y\subset\mathbb R$ which have the $\gamma$-space property with
 respect to countable Borel covers, whose product $X\times Y$ is not
 Menger, see \cite[Theorem~3.2]{MilTsaZso16}. Thus the product of
 spaces with the strongest combinatorial covering  property considered thus far
 might fail to have even the weakest one. This implies that no
 positive results for combinatorial covering properties can be obtained outright
 in ZFC.
  Unlike the vast majority of
 topological and combinatorial consequences under  CH, the latter
 one does not follow from any equality among cardinal
 characteristics of the continuum, see the discussion on \cite[p. 2882]{MilTsaZso16}.
However, there are many other negative results stating that under
certain equality among cardinal
 characteristics (e.g., $\mathit{cov}(\mathcal N)=\mathit{cof}(\mathcal N)$, $\hot b=\hot
d$, etc.\footnote{We refer the reader to \cite{Bla10} for the
definitions and basic properties of cardinal characteristics of the
continuum which are mentioned but are not used in the proofs in this
article.}) there are spaces $X,Y\subset \mathbb R$ with some combinatorial covering
property such that $X\times Y$ is not Menger, see, e.g.,
\cite{Bab09, RepZdo10, TsaSch02}.

Regarding the positive results,
 until recently the most
unclear situation was with the  Hurewicz  property and the weaker
ones. This was the main motivation for this article.
 There are two reasons why a product of Hurewicz spaces
$X,Y$ can fail to be  Hurewicz/Menger. In the first place, $X\times
Y$ may simply fail to be a Lindel\"of space, i.e., it might have an
open cover $\U$ without countable subcover. Then $X\times Y$ is not
even a Menger space. This may indeed happen: in ZFC there are two
normal spaces $X,Y$ with a covering property much stronger than the
Hurewicz one such that $X\times Y$ does not have the  Lindel\"of
property, see \cite[Section 3]{Tod95}. However, the above situation
becomes impossible if we restrict our attention to metrizable
spaces. This second case, on which we  concentrate in the sequel,
 turned out to be  sensitive to
the ambient set-theoretic universe: under CH there exists a Hurewicz space
whose square is not Menger, see \cite[Theorem~2.12]{COC2}. The above result has been achieved by
a transfinite construction of length $\w_1$, using the combinatorics of the ideal of
 measure zero  subsets of reals. This combinatorics turned out
 \cite[Theorem~43]{TsaSch02} to require
 much weaker set-theoretic assumptions  than CH.
 In particular, under the Martin
Axiom there are Hurewicz subspaces of the irrationals whose product is not Menger.

The following theorem, which is the main result of this article, shows
that an additional assumption in the results from
\cite{COC2,TsaSch02} mentioned above is really needed. In addition,
it implies that the affirmative answer to \cite[Problem~2]{COC2} is
consistent, see \cite[Section 2]{Tsa07}  for the discussion of this
problem.

\begin{theorem} \label{main}
In the Laver model for the consistency of the Borel's conjecture,
the product of any two Hurewicz spaces
has the Menger property provided that it is a Lindel\"of space.
In particular, the product of any
two Hurewicz metrizable spaces has the Menger property.
\end{theorem}

This theorem seems to be the first ``positive'' consistency result
related to the preservation by products
of combinatorial covering properties weaker than the $\sigma$-compactness,
in which  no further restrictions\footnote{The requirement
that the product must be Lindel\"of is vacuous for  metrizable spaces.
Let us note that nowadays the study of combinatorial covering properties
concentrates mainly on sets of reals.}
 on the spaces
 are assumed. The proof is based on the analysis of continuous maps
and names for reals in the  model of set theory constructed in \cite{Lav76}.
The question whether the product of Hurewicz metrizable spaces is a Hurewicz
space  in this model
remains open. It is worth mentioning here that in the Cohen model
there are Hurewicz subsets of $\mathbb R$ whose product has the  Menger property but
fails to have the  Hurewicz one,
see \cite[Theorem~6.6]{MilTsaZso16}.

 As suggested in its formulation, the model we use in
Theorem~\ref{main} was invented by Laver in order to prove that
Borel's conjecture is consistent, the latter being the  statement
that every strong measure zero set is countable. A strong measure
zero set is a subset $A$ of the real line with the following
property:
  for every sequence $\la \varepsilon_n:n\in\w\ra$ of positive reals there exists a sequence
$\la I_n:n\in\w\ra$ of intervals such that $|I_n| < \varepsilon_n$
for all $n$ and $A$ is covered by  the $I_n$'s. Here $|I_n|$ denotes
the length of $I_n$. Obviously, every countable set is a strong
measure zero set, and so is every union of countably many strong
measure zero sets. Sierpi\'nski proved in \cite{Sie28}  that  CH
implies the existence of uncountable strong measure zero sets, i.e.,
the negation of Borel's conjecture. Combined with this  Laver's
result gave the independence of Borel's conjecture. This outstanding
result was the first\footnote{According to our colleagues who worked
in set theory already in the 70s.} instance when a forcing, adding a
real, was iterated with countable supports without collapsing
cardinals. This work of Laver can be thought of as one of the
motivations behind   Baumgartner's axiom $A$ and later Shelah's
theory
 of proper forcing.

The conclusion of Theorem~\ref{main} does not follow from
 Borel's conjecture:
If we  add  $\w_2$ many random reals  over the Laver model then
Borel's conjecture still holds by \cite[Section 8.3.B]{BarJud95}  and  we have
$\mathit{cov}(\mathcal N)=\mathit{cof}(\mathcal N)$,
and hence in this model
there exists  a Hurewicz set of reals  whose square is not Menger, see \cite{TsaSch02}.
Thus  Borel's conjecture is consistent with the  existence of  a Hurewicz set of reals with nonMenger square.

Theorem~\ref{main} seems to be an instance of a more general
phenomena, namely that proper posets with fusion affect the behavior
of combinatorial covering properties. This happens because sets of reals with
certain combinatorial covering properties are forced  to have a rather clear
structure, which   suffices  to prove  positive preservation
results. For instance, the core of the proof of Theorem~\ref{main}
is that Hurewicz subspaces  of the real line are concentrated in a
sense around their ``simpler'' subspaces in the Laver model, see
Lemma~\ref{laver}.  As a consequence of corresponding structural
results we have proved \cite{Zdo??} that the Menger property is
preserved by finite products in the Miller model constructed in
\cite{Mil83}, and there are only $\hot c$ many Menger subspaces of
$\mathbb R$ in the Sacks model constructed in \cite{Sac??}, see
\cite{GarMedZdo??}.

 We believe that the interplay between forcing with fusion and
combinatorial covering properties has many more instances and it is worth
 considering whether there is some  deep reason behind it.

We assume that the reader is familiar with the basics of forcing as
well as with  standard proper posets used in the set theory of reals.

\section{Proof of Theorem~\ref{main}.  }

We shall first introduce a notion crucial for the proof of Theorem~\ref{main}.

\begin{definition}
 A topological space $X$ is called \emph{weakly concentrated}
if for every collection $\mathsf Q\subset [X]^\w$ which is cofinal with respect
to inclusion, and for every function $R:\mathsf Q\to\mathcal P(X)$
assigning to each $Q\in\mathsf Q$ a $G_\delta$-set $R(Q)$ containing $Q$,
there exists $\mathsf Q_1\in [\mathsf Q]^{\w_1}$ such that $X\subset\bigcup_{Q\in\mathsf Q_1}R(Q)$.
\end{definition}

The topology in $\mathcal P(\w)$  is generated by the countable base
$\B=\{[s,n]:s\in [\w]^{<\w}, n\in\w\}$, where $[s,n]=\{x\subset
\w:x\cap n=s\}$. Thus any open subset $O$ of $\mathcal P(\w)$ may be
identified with  $B_O=\{\la s,n\ra\in [\w]^{<\w}\times\w: [s,n]\subset
O\}$, and vice versa, any $B\subset [\w]^{<\w}\times\w$ gives rise
to an open $O_B=\bigcup_{\la s,n\ra\in B}[s,n]$. Note that $B\subset
B_{O_B}$ for all $B$.
 By a \emph{code} for an $F_\sigma$ subset $F$
of $\mathcal P(\w)$ we mean a sequence $\vec{B}=\la B_n:n\in\w\ra$
of subsets of $[\w]^{<\w}\times\w$ such that $F=\mathcal
P(\w)\setminus \bigcap_{n\in\w} O_{B_n}$. Obviously each $F_\sigma$
set $F\subset \mathcal P(\w)$ has many codes in the sense of the
definition above. For models $V\subsetneq V'$ of ZFC and   an
$F_\sigma$-subset $F\in V'$ of $\mathcal P(\w)$ we say that $F$
\emph{is coded in $V$} if there exists a  code for $F$ which is an
element of $V$. Note that being coded in $V$ doesn't imply being a
subset of $V$: $\mathcal P(\w)$ has codes in $V$ (e.g., $\la
\la\emptyset,0\ra:n\in\w\ra$) but it is not a subset of $V$ as long as
there are new reals in $V'$.

The consideration above also applies  to other Polish spaces having a base which can
be identified with some ``simple'' (e.g., constructive would suffice for our purposes)
 subset of $H(\w)$, the family of all hereditarily finite sets. Among them are $\w^\w, \mathcal P(\w)\times\w^\w$,
etc. In particular, since every continuous function from an $F_\sigma$-subset $F$ of $\mathcal P(\w)$
to $\w^\w$ is an $F_\sigma$-subset of $\mathcal P(\w)\times\w^\w$, we may speak about
such functions coded in $V$.

For a subset $X\in V'$ of $\mathcal P(\w)$ and an $F_\sigma$-subset
$Y$ of $X$ we shall say that \emph{$Y$ is coded in $V$} if there
exists an $F_\sigma$-subset $F$ of $\mathcal P(\w)$ coded in $V$
such that $Y=X\cap F$. Similarly, for continuous functions:
$f:Y\to\w^\w$ is coded in $V$ if there exists an $F_\sigma$-subset
$F$ of $\mathcal P(\w)$  such that $Y=X\cap F$, and a continuous
$\tilde{f}:F\to\w^\w$ coded in $V$, such that $f=\tilde{f}\uhr Y$.

The following lemma is the key part of the proof of
Theorem~\ref{main}. Its proof is reminiscent of that of
\cite[Theorem~3.2]{MilTsa10}.
 We will use the notation from \cite{Lav76} with  only differences being  that
smaller conditions in a forcing poset  are supposed to
 carry more information about the generic filter, and the ground model is denoted by $V$.

A subset $C$ of $\w_2$ is called an \emph{$\w_1$-club} if it is
unbounded and for every $\alpha\in\w_2$ of cofinality $\w_1$, if
$C\cap\alpha$ is cofinal in $\alpha$ then $\alpha\in C$.

\begin{lemma} \label{laver}
 In the Laver model every Hurewicz subspace of $\mathcal P(\w)$ is weakly concentrated.
\end{lemma}
\begin{proof}
We work in $V[G_{\w_2}]$, where $G_{\w_2}$ is $\IP_{\w_2}$-generic
and $\IP_{\w_2}$ is the iteration of length $\w_2$ with countable supports
of the Laver forcing, see \cite{Lav76} for details.

It is well known that a space $X\subset\mathcal P(\w)$ is Hurewicz if and only if
$f[X]$ is bounded with respect to $\leq^*$ for every continuous $f:X\to\w^\w$, see \cite[Theorem~4.4]{COC2}
or \cite{Hur27}.
 Let us fix a Hurewicz space $X\subset\mathcal P(\w)$.
The Hurewicz property is preserved by $F_\sigma$-subspaces
because it is obviously preserved by closed subspaces and countable unions.
Therefore
there exists an $\w_1$-club $C\subset \w_2$ such that for every $\alpha\in C$ and
continuous $f:F\to \w^\w$ coded in $V[G_\alpha]$, where $F$
is an $F_\sigma$-subspace  of $X$
coded in $V[G_\alpha]$,
there exists $b\in \w^\w\cap V[G_\alpha]$ such that $f(x)\leq^*b$
for all $x\in F$.
Indeed, since for every $\alpha<\w_2$ CH holds in $V[G_\alpha]$, there are at most $\w_1$ many
pairs $\la F,f\ra$ such that $F$
is an $F_\sigma$-subspace  of $X$
coded in $V[G_\alpha]$ and
 $f:F\to \w^\w$ is a continuous function  coded in $V[G_\alpha]$.
For every such pair find $\gamma_{F,f}<\w_2$ and $b_{F,f}\in\w^\w\cap V[G_{\gamma_{F,f}}]$
such that $f(x)\leq^*b_{F,f}$ for all $x\in F$. Let  $\gamma(\alpha)$  be the supremum of all
the $\gamma_{F,f}$ for $F,f$  as above. It is clear that the
$\w_1$-club $C\in V[G_{\w_2}]$ of all $\alpha$ such that $\gamma(\beta)<\alpha$ for all
$\beta<\alpha$ is as required.

Let $\mathsf Q\subset [X]^\w$ be cofinal with respect to the
inclusion. Fix a function $R:\mathsf Q\to\mathcal P(X)$ assigning to
each $Q\in\mathsf Q$ a $G_\delta$-subset $R(Q)$ of $\mathcal P(\w)$
containing $Q$. By a standard argument (see, e.g., the proof of
\cite[Lemma~5.10]{BlaShe87})
 there exists an $\w_1$-club $D\subset\w_2$
such that $\mathsf Q\cap V[G_\alpha]\in V[G_\alpha]$ and
$R\uhr (\mathsf Q\cap V[G_\alpha])\in V[G_\alpha]$ for\footnote{Here
 by $R$ we mean  the map
which assigns to a $Q\in \mathsf Q$ some code of $\mathcal P(\w)\setminus R(Q)$.}
all $\alpha\in D$. Moreover, using CH in the intermediate models as in the previous paragraph,
 we may also assume that  for every $Q_0\in [X\cap V[G_\alpha]]^\w\cap V[G_\alpha]$
there exists $Q\in \mathsf Q\cap V[G_\alpha]$ such that $Q_0\subset Q$.

Let us fix $\alpha\in C\cap D$.
We claim that $X\subset W$, where $W=\bigcup\{R(Q):Q\in\mathsf Q\cap V[G_\alpha]\}$.
Suppose that, contrary to our claim, there exists $p\in G_{\w_2}$ and a $\IP_{\w_2}$-name $\name{x}$
such that $p\forces\name{x}\in\name{X}\setminus\name{W}$.
By \cite[Lemma~11]{Lav76}
there is no loss of generality
in assuming that $\alpha=0$. Applying \cite[Lemma~14]{Lav76}
to a sequence $\la \name{a}_i:i\in\w\ra$ such that $\name{a}_i=\name{x}$ for all $i\in\w$,
we get a condition $p'\leq p$ such that $p'(0)\leq^0 p(0)$, and a finite set $U_s$
of reals for every $s\in p'(0)$ with $p'(0)\la 0\ra\leq s$, such that for each $\varepsilon>0$,
  $s\in p'(0)$ with $p'(0)\la 0\ra\leq s$, and for all but finitely many
immediate successors $t$ of $s$ in $p'(0)$ we have
$$  p'(0)_t\bigvid p'\uhr[1,\w_2)\forces \exists u\in U_s\: (|\name{x}-u|<\varepsilon). $$
Fix $Q\in\mathsf Q\cap V$ containing
$X\cap\bigcup\{U_s:s\in p'(0), s\geq p'(0)\la 0\ra\}$ and set $F=X\setminus R(Q)$.
Note that $F$ is an $F_\sigma$-subset of $X$ coded in $V$.
It follows that $p'\forces \name{x} \in \name{F}$ because $p'$ is stronger than $p$
that forces $\name{x}\not\in \name{W}\supset \name{X}\setminus\name{F}$. Consider the map
$f:F\to\w^S$, where $S=\{s\in p'(0): s\geq p'(0)\la 0\ra\}$, defined as follows:
$$f(y)(s)=[1/\min\{|y-u|:u\in U_s\}]+1$$
 for\footnote{Here $[a]$ is the largest integer not exceeding $a$.} all $s\in S$ and $y\in F$.
Since $F$ is disjoint from $Q$ which contains all the $U_s$'s, $f$
is well defined.
Since both $F$ and $f$ are coded in $V$, there exists $b\in\w^S\cap V$
such that $f(y)\leq^* b$ for all $y\in F$.

It follows from $p'\forces \name{x} \in \name{F}$ that $p'\forces \name{f}(\name{x}) \leq^* b$,
and  hence there exists $p''\leq p$ and a finite subset $S_0$ of $S$ such that
$p''\Vdash \name{f}(\name{x})(s) \leq b(s)$ for all $S\setminus S_0$.
By replacing $p''$ with  $p''(0)_s\vid p''\uhr[1,\w_2)$
for some $s\in p''(0)$, if necessary, we may additionally assume that $p''(0)\la 0\ra\in S\setminus S_0$.
Letting $s''=p''(0)\la 0\ra$, we conclude from the above that
 $p''\Vdash \name{f}(\name{x})(s'')\leq b(s'')$, which means that
$$ p''\Vdash     \min\{|\name{x}-u|:u\in U_{s''}\} \geq 1/b(s'').      $$
On the other hand, by our choice of $p'$ and $p''\leq p'$
we get that
for all but finitely many
immediate successors $t$ of $s''$ in $p''(0)$ we have
$$  p''(0)_t\bigvid p''\uhr[1,\w_2)\forces \exists u\in U_{s''}\: |\name{x}-u|<1/b(s'') $$
which means
$ p''(0)_t\bigvid p''\uhr[1,\w_2)\forces \min\{|\name{x}-u|:u\in U_{s''}\} <1/b(s'') $
and thus leads to a contradiction.
\end{proof}

A subset $X$ of $\mathcal P(\w)$ is called a \emph{$\lambda$-set} if
any  $A\in [X]^\w$ is a $G_\delta$-subset of $X$. Obviously, every
weakly concentrated $\lambda$-set has size $\leq\w_1$. Therefore
Lemma~\ref{laver} implies \cite[Theorem~3.2]{MilTsa10} because  the
property of a subset of $\mathcal P(\w)$ considered in the latter
theorem
 easily implies being both Hurewicz and a $\lambda$-set,
see, e.g., the proof of \cite[Theorem~5]{TsaZdo12} for details.

The next lemma can probably be considered as  folklore. We present
its proof for the sake of completeness.

\begin{lemma} \label{covering_g_delta}
Let $Y\subset \mathcal P(\w)$ be  Hurewicz  and $Q\subset\mathcal P(\w)$
 countable. Then for every $G_\delta$-subset $O$ of $\mathcal P(\w)^2$ containing
$Q\times Y$ there exists  a $G_\delta$-subset $R\supset Q$ such that
$R\times Y\subset O$.
\end{lemma}
\begin{proof}
 Without loss of generality we shall assume that $O$ is open. Let us write
$Q$ in the form $\{q_n:n\in\w\}$ and set $O_n=\{z\in\mathcal P(\w):\la q_n,z\ra\in O\}\supset
Y$. For every $n$ find a cover $\U_n$ of
$Y$ consisting of clopen subsets of $\mathcal P(\w)$ contained in $O_n$.
Let $\la\U'_k:k\in\w\ra$ be a sequence of open covers of $Y$ such that each $\U_n$
appears in it infinitely often. Applying the Hurewicz property of $Y$
we can find a sequence $\la \V_k:k\in\w\ra$ such that $\V_k\in [\U_k]^{<\w}$
and $Y\subset \bigcup_{k\in\w} Z_k$, where $Z_k=\bigcap_{m\geq k}\cup\V_m$.
Note that each $Z_k$ is compact and $Z_k\subset O_n$ for all $n\in\w $
(because there exists  $m\geq k$ such that $\U'_m=\U_n$,
and then $Z_k\subset\cup\V_m\subset O_n$). Thus $Q\times Y\subset Q\times (\bigcup_{k\in\w}Z_k)\subset O$.
Since $Z_k$ is compact, there exists for every $k$  an open $R_k\supset Q$
such that $R_k\times Z_k\subset O$. Set $R=\bigcap_{k\in\w}R_k$ and note that
$R\supset Q$ and $R\times Y\subset R\times\bigcup_{k\in\w}Z_k\subset O$.
\end{proof}

Let $A$ be a countable  set and $x,y\in\w^A$. As usually,
$x\leq^* y$ means  that $\{a\in A: x(a)>y(a)\}$ is finite.
The smallest cardinality of an unbounded with respect to
$\leq^*$ subset of $\w^\w$ is denoted by  $\hot b$.
It is well known that $\w_1<\hot b$ in the Laver model, see \cite{Bla10}
for this fact as well as systematic treatment of cardinal characteristics of
reals.

The second part of Theorem~\ref{main} is a direct consequence of Lemma~\ref{laver}
and the following

\begin{proposition} \label{main1}
 Suppose that $\hot b>\w_1$. Let $Y \subset \mathcal P(\w)$ be a Hurewicz space
and $X\subset \mathcal P(\w)$  weakly concentrated. Then $X\times Y$ is Menger.
\end{proposition}
\begin{proof}
Fix a sequence $\la \U_n:n\in\w\ra$ of  covers of $X\times Y$
by clopen subsets of $\mathcal P(\w)^2$. For every
 $Q\in [X]^\w $
fix a sequence
$\la \W^{Q}_n:n\in \w\ra$  such that $\W^{Q}_n\in [\U_n]^{<\w}$ and
$Q\times Y\subset\bigcap_{n\in \w}\bigcup_{m\geq n}\cup\W^{Q}_m$.
Letting $O_{Q}=\bigcap_{n\in \w}\bigcup_{m\geq n}\cup\W^{Q}_m$ and using Lemma~\ref{covering_g_delta},
we can find  a $G_\delta$-subset $R_{Q}\supset Q$  such that
$R_{Q}\times Y\subset O_{Q}$. Since $X$ is weakly concentrated, there exists
$\mathsf Q\subset [X]^\w$ of size $|\mathsf Q|=\w_1$
such that $ R =\bigcup\{ R_{Q}  : Q\in\mathsf Q\} $ contains $X$ as a subset.
Let us fix $x\in X$ and find
 $Q\in\mathsf Q$  such that $x\in R_{Q}$.
Then $\{x\}\times Y \subset R_{Q}\times Y \subset O_Q$.
Therefore for every $\la x,y\ra \in X\times Y$
there exists $Q\in\mathsf Q$ such that $\la x,y\ra\in O_{Q}=\bigcap_{n\in \w}\bigcup_{m\geq n}\cup\W^{Q}_m$.
Let us write $\U_n$ in the form $\{U^n_k:k\in\w\}$
and for every $Q\in\mathsf Q$
 fix   a real $b_Q\in\w^\w$ with the property $\W^Q_n\subset\{U^n_k:k\leq b_Q(n)\}$.
Since $|\mathsf Q|=\w_1<\hot b$, there exists $b\in\w^\w$ such that $b_Q\leq^* b$
for all $Q\in\mathsf Q$.
It follows from the above that $X\times Y\subset\bigcup_{n\in\w}\bigcup_{k\leq b(n)}U^n_k$,
which completes our proof.
\end{proof}

A family $\F\subset[\w]^\w$ is called a \emph{semifilter} if for every
$F\in\F$ and $X\subset \w$, if $|F\setminus X|<\w$
then $X\in\F$.

The proof of  the first part of Theorem~\ref{main}  uses  characterizations
of the properties of Hurewicz and Menger obtained in \cite{Zdo05}.
Let  $u=\la U_n : n\in\omega\ra$ be a sequence of subsets of a set $X$.
For every $x\in X$ let  $I_s(x,u,X)=\{n\in\omega:x\in U_n\}$. If every
$I_s(x,u,X)$ is infinite (the collection of all such sequences $u$ will be denoted
by $\Lambda_s(X)$), then we shall denote by $\mathcal U_s(u,X)$
the smallest semifilter on $\omega$ containing all $I_s(x,u,X)$.
By \cite[Theorem~3]{Zdo05},  a Lindel\"of topological space $X$  is Menger (Hurewicz) if and only if
for every  $u \in\Lambda_s(X)$ consisting of open sets,
  the semifilter $\mathcal U_s(u,X)$ is Menger (Hurewicz).
The proof given there also works if we consider only those
$\la U_n : n\in\omega\ra\in\Lambda_s(X)$ such that all $U_n$'s belong to a given base of
$X$.

\medskip

\noindent\textit{Proof of Theorem~\ref{main}.} \
Suppose that $X,Y$ are Hurewicz spaces such that $X\times Y$ is Lindel\"of
and fix  $w=\la U_n\times V_n :n\in\w\ra\in\Lambda_s(X\times Y)$
consisting of open sets.
Set  $u=\la U_n:n\in\w\ra$,  $v=\la V_n:n\in\w\ra$, and note that
$u\in\Lambda_s(X)$ and $v\in\Lambda_s(Y)$.
It is easy to see
that
$$\U_s(w,X\times Y)=\{A\cap B: A\in \U_s(u,X), B\in \U_s(v,Y)\},$$
and hence $\U_s(w,X\times Y)$ is a continuous image of
$\U_s(u,X)\times \U_s(v,Y)$. By \cite[Theorem~3]{Zdo05} both of
latter ones are Hurewicz, considered as subspaces of $\mathcal
P(\w)$, and hence their product is a Menger space by
Proposition~\ref{main1} and Lemma~\ref{laver}. Thus $\U_s(w,X\times
Y)$ is Menger, being a continuous image of a Menger space. It now
suffices to use \cite[Theorem~3]{Zdo05} again. \hfill $\Box$
\medskip

\noindent \textbf{Acknowledgments.} We would like to thank Marion
Scheepers for fruitful discussions of products of Hurewicz spaces in
various models of Borels's conjecture. We also thank Boaz Tsaban, as
well as the referees,  for many comments and suggestions which
improved our presentation. The first author was partially supported by
the Slovenian Research Agency grant P1-0292. The second author would
like to thank  the Austrian Science Fund FWF (Grants I 1209-N25 and
I 2374-N35)
 for generous support for this research.

\end{document}